\documentclass[12pt]{amsart}%
\usepackage{bm,graphics,epsfig,color,float,latexsym,array,amsthm,mathrsfs,amssymb,amsfonts,graphicx,amsmath,psfrag,enumerate}%
\usepackage{epic,fancybox}
\usepackage[usenames,dvipsnames]{pstricks}
 \usepackage{pst-grad} 
 \usepackage{pst-plot} 
\numberwithin{equation}{section}
\newtheorem{thm}{Theorem}[section]

\newtheorem{lem}[thm]{Lemma}
\newtheorem{defn}[thm]{Definition}

\newcommand{\sinv}{\piinv_\text{b}}
\newcommand{\s}{\pi_\text{b}}
\newcommand{\sbar}{\bar\pi_\text{b}}
\newcommand{\sinvbar}{\bar\pi^{-1}_\text{b}}

\newcommand{\ra}{\rightarrow}
\newcommand{\wt}{\tilde}

\newcommand{\Balg}{\mathscr{B}}

\newcommand{\A}{\mathcal{A}}
\newcommand{\piinv}{\pi^{-1}}

\newtheorem*{ack*}{Acknowledgment}

\begin{document}

\title{Computing degree and class degree}
\author{Mahsa Allahbakhshi}

\address{Mahsa Allahbakhshi, Centro de Modelamiento Matematico, Universidad de Chile,
Av. Blanco Encalada 2120, Piso 7, Santiago de Chile}
\email{mallahbakhshi@dim.uchile.cl}

\thanks{The author was supported by FONDECYT project 3120137.}
\maketitle
\begin{abstract}
Let $\pi$ be a factor code from a one dimensional shift of finite type $X$ onto an irreducible sofic shift $Y$. If $\pi$ is finite-to-one then the number of preimages of a typical point in $Y$ is an invariant called the degree of $\pi$. In this paper we present an algorithm to compute this invariant. The generalized notion of the degree when $\pi$ is not limited to finite-to-one factor codes, is called the class degree of $\pi$. The class degree of a code is defined to be the number of transition classes over a typical point of $Y$ and is invariant under topological conjugacy. We show that the class degree is computable.
\end{abstract} 

\section{Introduction}
One source of inspiration in symbolic dynamics comes from storage systems and transmission in computer science. For example sofic shifts are analogous to regular languages in automata theory, so a sofic shift and its cover are natural models for information storage and transmission. As a result, starting with a presentation of a dynamical system, there are known algorithms constructed to compute some kind of object from such presentation. Given a sofic shift, Coven and Paul constructed a finite procedure to obtain a finite-to-one sofic cover~\cite{CovP77}. There is an algorithm to determine whether two graphs present the same sofic shift~\cite{lm}. Kim and Roush showed that the shift equivalence of sofic systems is decidable~\cite{KimR90}.

In this work, starting from a sofic shift and its finite-to-one cover, we present an algorithm to compute the number of preimages of a typical point of the sofic. Moreover, we show that in the case of having an infinite-to-one cover, an analogous object can be computed in finitely many steps. 

Given a factor code $\pi$ from a one-dimensional shift of finite type $X$ to a sofic shift $Y$, when $\pi$ is finite-to-one there is a quantity assigned to $\pi$ called the \emph{degree} of $\pi$. The degree of a finite-to-one code is defined to be the minimal number of $\pi$-preimages of the points in $Y$. One can show that the number of preimages of every transitive point in $Y$ is exactly the degree of $\pi$. The degree of a finite-to-one code is widely-studied and known to be invariant under recoding~\cite{lm}. In the first section of this paper we present an algorithm to compute this invariant.

When $\pi:X\to Y$ is not limited to be finite-to-one an analogous of the degree, called the \emph{class degree}, is defined to be the minimal number of transition classes (always finite) over the points in $Y$. The definition of a transition class is motivated by communicating classes in Markov chains. Roughly speaking, two preimages $x$ and $\bar x$ of a point $y$ in $Y$ lie in the same equivalence class, \emph{transition class}, if one can find a preimage $z$ of $y$ which is equal to $x$ up to an arbitrarily large given positive coordinate and right asymptotic to $\bar x$ and vice versa. When $\pi$ is finite-to-one then the degree and the class degree of $\pi$ match. One can also show that the class degree is invariant under topological conjugacy and the number of transition classes over any transitive point of $Y$ is exactly the class degree of $\pi$. One of the main applications of the class degree is bounding the number of measures of relative maximal entropy~\cite{aq}. Such 
measures have applications in 
information theory, computing Hausdorff dimensions and functions of Markov chains \cite{black,BoyTun,BurRos,GatPerexp,karlinfo}. In the second section of this paper we show that the class degree is computable.

\section{Background}
Throughout this paper, $X$ is a one-dimensional shift of finite type (SFT) with the shift transformation $T$. The alphabet of $X$ is denoted by $\A(X)$ and the $\sigma$-algebra on $X$ generated by cylinder sets is denoted by $\Balg_X$.  A triple $(X,Y,\pi)$ is called a \textbf{factor triple} when $\pi:X\to Y$ is a continuous shift-commuting map (factor code) from an SFT $X$ onto a subshift $Y$ (sofic shift $Y$). When $\pi$ is a one-block factor code induced by a symbol-to-symbol map $\s:\A(X)\to\A(Y)$ we naturally extend $\s$ to blocks in $\Balg_X$ (b stands for block). 
When $\pi$ is a finite-to-one factor code there is a uniform upper bound on the number of
pre-images of points in $Y$. The minimal number of pre-images of points in $Y$ is called the \textbf{degree} of the code and is denoted by $d_\pi$.

\begin{defn}
 We say two factor triples $(X,Y,\pi)$ and $(\wt X,\wt Y,\wt\pi)$ are \textbf{conjugate} if $X$ is conjugate to $\wt X$ under a conjugacy $\phi$, $Y$ is conjugate to $\wt Y$ under a conjugacy $\psi$, and $\wt\pi\circ\phi=\psi\circ\pi$. 
\end{defn}

\begin{thm}~\cite{lm}\label{thm:magicsymbol}
 Let $(X,Y,\pi)$ be a factor triple. There is a factor
triple $(\wt{X},\wt{Y},\wt\pi)$ conjugate to
$(X,Y,\pi)$ such that $\wt{X}$ is one-step and $\tilde\pi$ is one-block. 
\end{thm}

\begin{thm}~\cite{lm}\label{invdeg}
 Given two conjugate factor triples $(X,Y,\pi)$ and $(\wt{X},\wt{Y},\wt\pi)$, we have $d_\pi=d_{\wt\pi}$.
\end{thm}

\begin{thm}\label{thm:samenumber}~\cite{lm}
Let $\pi$ be a finite-to-one factor code from an SFT $X$ onto an
irreducible sofic shift $Y$. Then every transitive point of $Y$ has exactly $d_{\pi}$
preimages. 
\end{thm}
Given a one-block factor code $\pi$, above every $Y$-block $W$ there is a set of $X$-blocks $U$ which are sent to $W$ by $\s$; i.e., $\s(U)=W$. Given $0\leq i <|W|$, define
$$\sinv(W)_i=\{a\in\A(X)\colon \exists W' \text{ with }\s(W')=W,\,W'_i=a\}$$ and $$d^*_{\pi}=\min\{|\sinv(W)_i|:W\in\mathscr L(Y),\, 0\leq i< |W|\}.$$ 

\begin{thm}\label{thm:degmag}\cite{lm}
Let $\pi$ be a finite-to-one one-block factor code from an SFT $X$ onto an irreducible sofic shift $Y$. Then $d^*_{\pi}=d_\pi.$ 
\end{thm}

Given a one-block factor code $\pi:X\to Y$, a \textbf{magic block} is a block $W$ such that
$d(W,i)=d^*_{\pi}$  for some $0\leq i<|W|$. Such an index $i$ is
called a \textbf{magic coordinate} of $W$. A factor code $\pi$ has a
\textbf{magic symbol} if there is a magic block of $\pi$ of length $1$. 

The class degree defined below is a quantity analogous to the degree which is defined in the general case when $\pi$ is not only limited to be finite-to-one. 
\begin{defn}\label{defn:equiv}
Suppose $(X,Y,\pi)$ is a factor triple and $x,\,x'\in X$. There is a \textbf{transition} from $x$ to
$x'$ denoted by $x\to x'$ if for each integer $n$, there is a point $v$ in $X$ so that
\begin{enumerate}
\item $\pi(v)=\pi(x)=\pi(x')$, and
\item $v_{-\infty}^n=x_{-\infty}^n,~~ v^{\infty}_i=x'^{\infty}_i$ for some $i\geq n$.
\end{enumerate}
 Write $x\sim x'$, and say $x$ and $x'$ are in the same (equivalence) \textbf{transition class} if $x\ra x'$ and $x'\ra x$. The minimal number of transition classes over points of $Y$ is called the \textbf{class degree} of $\pi$ and denoted by $c_\pi$.

\end{defn}

\begin{thm}\label{invclass}~\cite{aq}
 Given two conjugate factor triples $(X,Y,\pi)$ and $(\wt{X},\wt{Y},\wt\pi)$, we have $c_\pi=c_{\wt\pi}$.
\end{thm}

\begin{thm}\label{thm:classdegree}~\cite{aq}
Let $\pi$ be a one-block factor code from a one-step SFT $X$ to a sofic shift $Y$. The number of transition classes over a right transitive point of $y$ is exactly the class degree.
\end{thm}

Theorem~\ref{thm:classdegree}, in below, provides a finitary characterization of the class degree.

\begin{defn}\label{defn:TB}
Let $\pi:X\ra Y$ be a one-block factor code from a one-step SFT $X$ to a
sofic shift $Y$ and let $W=W_0\dots W_p$ be a block of $Y$. Let $0<n<p$ and let $M$ be a subset of $\piinv_\text{b}(W)_n$. We say $U\in\piinv_\text{b}(W)$ is \textbf{routable} through $a\in M$ at time $n$ if there is a block $U'\in\piinv_\text{b}(W)$ with
$U'_0=U_0,~U'_{n}=a$, and $U'_{p}=U_{p}$. A triple $(W,n,M)$ is called a
\textbf{transition block} of $\pi$ if every block in $\piinv_\text{b}(W)$ is routable through a symbol of $M$ at time $n$. The cardinality of the set $M$ is called the \textbf{depth} of the
transition block $(W,n,M)$.

 Let
$$c^*_{\pi}=\min\{|M|\colon (W,n,M)\textrm{ is a transition block of $\pi$}\}.$$
A \textbf{minimal transition block} of $\pi$ is a
transition block of depth $c^*_\pi$.
\end{defn}

\begin{thm}\label{thm:classdegree}~\cite{aq}
Let $\pi$ be a one-block factor code from a one-step SFT $X$ to a sofic shift $Y$. Then $c^*_{\pi}=c_\pi$.
\end{thm}
The following theorem shows the relation between the degree and the class degree of a finite-to-one factor code.
\begin{thm}\cite{aq}
Let $\pi:X\ra Y$ be a finite-to-one factor code from a SFT $X$ to an
irreducible sofic shift $Y$. Then $c_\pi=d_\pi$.
\end{thm} 

\section{Degree algorithm} \label{sec:degalg}
In this section we present an algorithm to compute the degree of a finite-to-one factor code. By Theorems~\ref{thm:magicsymbol} and~\ref{invdeg}, without loss of generality, we may assume $\pi$ is a one-block factor code defined on a one-step SFT.

Let $X$ be a one-step SFT with alphabet $\A(X)=\{a_1,\dots,a_i\}$ and adjacency matrix $I$. Let $\pi:X\to Y$ be a finite-to-one one-block factor code from $X$ to a sofic shift $Y$ with the alphabet $\A(Y)=\{b_1,\dots,b_j\}$. Make two graphs $G$ and $G'$ as follows. 

Let $G$ and $G'$ both have the same vertex set $$\mathcal V=\bigcup_{b\in\A(Y)}\mathscr P \{\sinv (b)\}=\{A_1,\dots,A_m\}$$ where $\mathscr P\{\sinv (b)\}$ stands for the power set of $\{\sinv (b)\}$. For what we need later, divide the vertex set $\mathcal V$ into two parts $\mathcal U$ and $\mathcal V-\mathcal U$ where $\mathcal U=\{\sinv (b)\colon b\in\A(Y)\}$. Form the adjacency matrix $M$ of $G$ and the adjacency matrix $M'$ of $G'$ as follows. Let $A,A'\in\mathcal V$, then $A\subseteq\sinv(b)$ and $A'\subseteq\sinv(b')$ for some $b,b'$ in $\A(Y)$. Say $M_{AA'}=1$ if the following conditions hold.
\begin{enumerate}
\item $bb'$ is a block of $Y$,
\item $A'$ is exactly the set of all symbols $a'$ in $\sinv(b')$ such that $aa'$ is a block of $X$ for some $a$ in $A$.
\end{enumerate}
Otherwise $M_{AA'}=0$. Say $M'_{AA'}=1$ if we have
\begin{enumerate}
\item $bb'$ is a block of $Y$,
\item $A$ is exactly the set of all symbols $a$ in $\sinv(b)$ such that $aa'$ is a block of $X$ for some $a'$ in $A'$.
\end{enumerate}
Otherwise $M'_{AA'}=0$. Consider the following subsets of the vertex set $\mathcal V$,
$$S=\{A\in\mathcal V \colon \text{ there is a finite path in } G \text{ from }B \text{ to } A \text{ for some }B\in\mathcal U\},$$ and 
$$S'=\{A\in\mathcal V \colon\text{ there is a finite path in } G' \text{ from } A \text{ to } B \text{ for some }B\in\mathcal U\}.$$ 
\begin{thm}~\label{thm:degalg}
Using above notations, we have $$d_\pi=\min_{\substack{A\in S\\A'\in S'}}\{|A\cap A'|\colon A\cap A'\neq\emptyset\}.$$
\end{thm}
\begin{proof}
Note that $G$ and $G'$ are finite directed graphs. Let $X_G$ and $X_{G'}$ be the shift spaces represented by $G$ and $G'$ accordingly. Let $\sbar$ be the map from $\mathcal V$ to $Y$ taking $A\in\mathcal V$ to $\pi(a)$ for some $a\in A$ (note that $\sbar(A)$ is independent of $a\in A$). Then $\sbar$ induces a one-block code $\bar\pi_G:X_G\to Y$ and a one-block code $\bar\pi_{G'}:X_{G'}\to Y$.\\

The key feature of these two graphs is the following. For any block $$W=W_0\dots W_k$$ of $Y$ there is a unique walk $$U=U_0\dots U_k$$ in $G$ with the following properties.
\begin{enumerate} 
\item $U_0=\sinv(W_0)=\sinvbar(W_0)$.
\item $\bar\pi_G(U)=W$.
\item $U_k=\{a\in\A(X)\colon \exists B\in\piinv(W)\text{ and }B_k=a\}$.
\end{enumerate}
 
 Similarly there is a unique walk $$V=V_0\dots V_k$$ in $G'$ with the following properties.
\begin{enumerate}
\item $V_k=\sinv(W_k)=\sinvbar(W_k)$.
\item $\bar\pi_{G'}(V)=W$.
\item $V_0=\{a\in\A(X)\colon \exists\, B\in\piinv(W)\text{ and }B_0=a\}$.
\end{enumerate}
Let $A\in S$, $A'\in S'$, and $A\cap A'\neq\emptyset$. Let $U=U_0\dots U_c$ be a walk in $G$ where $U_0\in\mathcal U$ and $U_c=A$. Let $V=V_0\dots V_k$ be a walk in $G'$ where $V_0=A'$ and $V_k\in\mathcal U$. Since $A\cap A'\neq\emptyset$ we have $\bar\pi_G(A)=\bar\pi_{G'}(A')$, and consequently $\bar\pi_G(U)\bar\pi_{G'}(V)=W_0\dots W_{c+k}=W$ is a block of $Y$. By the key feature discussed in the above paragraph, we have $$A\cap A'=\{a\in\A(X)\colon W'\in\sinv(W),\text{ and }W'_c=a\}.$$
Since by definition $$d_\pi=\min_{\substack{W\text{ block of }Y\\ 0\leq i\leq|W|}}|\{a\in\A(X)\colon \exists\, W'\in\piinv(W), W'_i=a\}|,$$ we have $$d_\pi\leq \min_{\substack{A\in S\\A'\in S'}}\{|A\cap A'|\colon A\cap A'\neq \emptyset\}.$$
Now we prove the inequality in the other direction; that is we seek to show that $$d_\pi\geq \min_{\substack{A\in S\\A'\in S'}}\{|A\cap A'|\colon A\cap A'\neq \emptyset\}.$$
Let $W=W_0\dots W_k$ be a magic block of $\pi$ with a magic coordinate $0\leq c\leq k$. Let $$D=\{a\in\A(X)\colon W'\in\sinv(W),\,W'_c=a\}.$$ Consider $U=W_0\dots W_c$ and $V=W_c \dots W_k$. Let $\bar U=\bar U_0\dots \bar U_c$ be the unique path in $G$ which maps to $U$ and $\bar U_0=\sinv(W_0)$. Then $$\bar U_C=\{a\in\A(X)\colon F\in\sinv(U),\,F_c=a\}.$$ Let $\bar V=\bar V_c\dots \bar V_k$ be a unique path in $G'$ which maps to $V$ and $\bar V_k=\sinv(W_k)$. Then $$\bar V_c=\{a\in\A(X)\colon G\in\sinv(V),\,G_0=a\}.$$ It follows that $$\bar U_c\cap\bar V_c=\{a\in\A(X)\colon W'\in\sinv(W),\,W'_c=a\}=D.$$ Therefore $$d\geq \min_{\substack{A\in S\\A'\in S'}}\{|A\cap A'|\colon A\cap A'\neq \emptyset\},$$ which completes the proof.
\end{proof}

\section{Class degree is computable}
In this section we find an upper bound on the length of a minimal transition block of a factor triple. Then by Theorem~\ref{thm:classdegree} it follows that the class degree of a factor code is computable. Without loss of generality, by Theorem~\ref{thm:magicsymbol} and Theorem~\ref{invclass} we may assume that the SFT is one-step and the factor code is one-block.

\begin{thm}\label{thm:length}
 Let $\pi:X\to Y$ be a one-block factor code from a one-step SFT $X$ to a sofic shift $Y$. Let $f=\max\{|\sinv(w)|\colon w\in\A(Y)\}$. There is a minimal transition block $(W,n,M)$ of $\pi$ with $$|W|\leq |\A(Y)|\times 2^{f^2+f+1}.$$
\end{thm}

We need the following definitions and lemmas to prove Theorem~\ref{thm:length}.

\begin{defn}\label{defn:pairs}
Let $\pi:X\to Y$ be a one-block factor code from a one-step SFT $X$ to a sofic shift $Y$. Given $A,\,B\subseteq\A(X)$ and $\gamma\in\mathscr P(A\times B)$ say $A$ \textbf{pairs with} $B$ {\bf in form} $\gamma$ when there is a block $W=W_0\dots W_n$ of $Y$ with $\sinv(W)_0=A$, $\sinv(W)_n=B$ such that $(a^*,b^*)\in\gamma$ if and only if there is $I\in\sinv(W)$ which begins at $a^*$ and ends at $b^*$; that is, $I_0=a^*$ and $I_n=b^*$.

\end{defn}

Note that $A$ can pair with $B$ in at most $2^{|A\times B|}$ distinct forms. 
\begin{defn}\label{defn:w-pair}
Let $\pi:X\to Y$ be a one-block factor code from a one-step SFT $X$ to a sofic shift $Y$. Given $A,\,B\subseteq\A(X)$ and a block $W=W_0\dots W_n$ of $Y$ with $\sinv(W)_0=A$, $\sinv(W)_n=B$, a $W$-\textbf{paring of} $A,B$, denoted by $P_{W}(AB)$, is the set $\gamma\in\mathscr P(A\times B)$ which contains all pairs $(a^*,b^*)$ such that there is a block in $I\in\sinv(W)$ with $I_0=a^*$ and $I_n=b^*$.
\end{defn}

\begin{lem}\label{cl:join}
  Let $\pi:X\to Y$ be a one-block factor code from a one-step SFT $X$ to a sofic shift $Y$. Let $A,B,C\subseteq\A(X)$. Let $D=D_0\dots D_l$ and $D'=D'_0\dots D'_{l'}$ be two blocks of $Y$ with $\sinv(D)_0=\sinv(D')_0=A$, $\sinv(D)_l=\sinv(D')_{l'}=B$, and $P_{D}(AB)=P_{D'}(AB)$. Let $E=E_0\dots E_s$ and $E'=E'_0\dots E'_{s'}$ be two blocks of $Y$ with $\sinv(E)_0=\sinv(E')_0=B$, $\sinv(E)_{s}=\sinv(E')_{s'}=C$, and $P_{E}(BC)=P_{E'}(BC)$. Then we have $P_{W}(AC)=P_{W'}(AC)$ where $W$ and $W'$ are blocks of $Y$ formed by joining $E$ to $D$ and $E'$ to $D'$ as follows: $W=D_0\dots D_lE_1\dots E_s$ and $W'=D'_0\dots D'_{l'}E'_1\dots E'_{s'}$.
\end{lem}

\begin{proof}
 First note that $D_l=E_0$ and $D'_{l'}=E'_0$ and therefore $W$ and $W'$ are legal blocks of $Y$ with $|W|=l+s+1$ and $|W'|=l'+s'+1$. It is enough to show $P_{W}(AC)\subseteq P_{W'}(AC)$. Let $(a^*,c^*)$ be in $P_W(AC)$. This means there is a block $I$ in $\sinv(W)$ which starts at $a^*$ and ends at $c^*$. Since $B=\sinv(W)_{l}$ we have $I_{l}\in B$. Denote $I_{l}$ by $b^*$. It follows that $(a^*,b^*)\in P_D(AB)$ and $(b^*,c^*)\in P_E(BC)$. Hence by assumption, $(a^*,b^*)\in P_{D'}(AB)$ and $(b^*,c^*)\in P_{E'}(BC)$. This means there is a block $U\in\sinv(D')$ with $U_0=a^*$ and $U_{l'}=b^*$, and a block $V\in\sinv(E')$ with $V_0=b^*$ and $V_{s'}=c^*$. Form the block $J$ by joining the blocks $V$ to $U$ as below; $J=U_0\dots U_{l'}V_1\dots V_{s'}$. Clearly $J\in\sinv(W')$. Since $J_0=a^*$ and $J_{l'+s'+1}=c^*$ it follows that $(a^*,c^*)\in P_{W'}(AC)$.
\end{proof}

\begin{lem}\label{cl:transblock}
 Let $(W,n,M)$ be a transition block with $|W|=t+1$. Let $\sinv(W)_0=A$, $\sinv(W)_n=B$, and $\sinv(W)_t=C$ for some $A,B,C\subseteq\A(X)$. Let $W'$ with $|W'|=t'+1$ be another block of $Y$ with $\sinv(W)_0=A$, $\sinv(W)_{n'}=B$ for some $0\leq n<t'$, and $\sinv(W)_{t'}=C$ such that $P_{W_0\dots W_{n}}(AB)=P_{W'_0\dots W'_{n'}}(AB)$ and $P_{W_n\dots W_{t}}(BC)=P_{W'_{n'}\dots W'_{t'}}(BC)$. Then $(W',n',M)$ is a transition block of $\pi$. 
\end{lem}

\begin{proof}
 Note that by Lemma~\ref{cl:join}, since $P_{W_0\dots W_{n}}(AB)=P_{W'_0\dots W'_{n'}}(AB)$ and $P_{W_n\dots W_{t}}(BC)=P_{W'_{n'}\dots W'_{t'}}(BC)$, we have $P_W(AC)=P_{W'}(AC)$.
 Suppose $M={b_1,\dots,b_i}$. Let $U\in\sinv(W')$, we need to show that $U$ is routable through a member of $M$. Let $U_0=a^*$ and $U_{t'}=c^*$. Since $(a^*,c^*)\in P_{W'}(AC)$ we have $(a^*,c^*)\in P_W(AC)$. It follows that there is a block $V\in\sinv(W)$ with $V_0=a^*$ and $V_t=c^*$. Since $(W,n,M)$ is a transition block, then block $V$ must be routable through some symbol $b^*\in M$ which implies that $(a^*,b^*)\in P_{W_0\dots W_n}(AB)=P_{W'_0\dots W'_{n'}}(AB)$, and $(b^*,c^*)\in P_{W_n\dots W_t}(BC)=P_{W'_{n'}\dots W'_{t'}}(BC)$. It follows that there is a block $U'\in\sinv(W')$ with $U'_0=a^*$, $U'_{n'}=b^*$, and $U'_{t'}=c^*$, meaning that $U$ is routable through $b^*\in M$.
\end{proof}

\begin{proof}[Proof of Theorem~\ref{thm:length}]
Let $(W,n,M)$ be a minimal transition block of $\pi$ with $|W|=t+1$. Assume $t+1>|\A(Y)|\times 2^{f^2+f+1}$. We construct a transition block $(W',n',M')$ with $|W|\leq |\A(Y)|\times 2^{f^2+f+1}$, but yet $M'=M$. Such transition block has depth $|M|$ and therefore is a minimal transition block. 

Let $\sinv(W)_0=A$, $\sinv(W)_n=B$. Denote $P_{W_0\dots W_n}(AB)$ by $\gamma$. We show that $A$ can pair with $B$ in form $\gamma$ in less than or equal $|\A(Y)|\times 2^{f^2+f}$ numbers of steps.
 
Recall that given any $d\in\A(Y)$ there are at most $2^f$ number of distinct subsets of $\sinv(d)$. Moreover, given $D\subseteq\sinv(d)$, $A$ can pair with $D$ in at most $2^{f^2}$ distinct forms. Therefore, since $n> |\A(Y)|\times 2^{f^2+f}$, 
there is at least one symbol $d\in\A(Y)$ and a subset $D\subseteq\sinv(d)$ such that $d$ occurs in $W$ at two distinct positions $W_k$ and $W_r,k<r\leq n$, with $\sinv(W)_k=\sinv(W)_r=D$, and $P_{W_0\dots W_k}(AD)=P_{W_0\dots W_r}(AD)$. Let $Z=Z_0\dots Z_s$ be the block $W_0\dots W_kW_{r+1}\dots W_n$. Note that by assumption, $W_k=W_r$ which implies that $Z$ is a legal block of $Y$. If $s\leq |\A(Y)|\times 2^{f^2+f}$ then we are done, if not, we repeat the process until $A$ pairs with $D$ in form $\gamma$ in less than $|\A(Y)|\times 2^{f^2+f}$ number of steps. The result is a block $G=G_0\dots G_{n'}$ of $Y$ where $n'\leq |\A(Y)|\times 2^{f^2+f}$, $\sinv(G)_0=A$, $\sinv(G)_{n'}=B$, and $P_{G_0\dots G_{n'}}(AB)=\gamma$.

Now let $\sinv(W)_t=C$, and denote $P_{W_n\dots W_t}(BC)$ by $\delta$. By similar argument $B$ can pair with $C$ in less than or equal $|\A(Y)|\times 2^{f^2+f}$ number of steps; that is, there is a block $H=H_{n}\dots H_{m}$ of $Y$ where $m-n+1\leq |\A(Y)|\times 2^{f^2+f}$, $\sinv(H)_{n}=B$, $\sinv(H)_{m}=C$, and $P_{H_{n}\dots H_{m}}(BC)=\delta$. 

Make the new block $W'=W'_0\dots W'_{t'}$ by joining the two blocks $G$ and $H$ as follows: $W'=G_0\dots G_{n'}H_{n+1}\dots H_{m}$. Note that since $G_{n'}=H_{n}$ the block $W'$ is a legal block of $Y$. Moreover, we have $|W'|\leq |\A(Y)|\times 2^{f^2+f+1}$. Now it is easy to see that $W'$ is a minimal transition block. Note that $\sinv(W')_0=A$, $\sinv(W')_{n'}=B$, $\sinv(W')_{t'}=C$, $P_{W'_0\dots W'_{n'}}(AB)=\gamma$, and $P_{W'_{n'}\dots W'_{t'}}(BC)=\delta$. Lemma~\ref{cl:transblock} implies that $(W',n',M)$ is a transition block of $\pi$, and since its depth is $|M|$ we conclude that it is a minimal transition block of $\pi$. 
\end{proof}

\section{Open question}
 Given a factor triple $(X,Y,\pi)$ since $c_\pi\leq d_\pi$, if $d_\pi=1$ then $c_\pi=1$. Thus the algorithm given in Section~\ref{sec:degalg} will compute the class degree as well. We can actually use the same two graphs $G$ and $G'$ and modify Theorem~\ref{thm:degalg} by applying the upper bound on the length of a minimal transition block stated in Theorem~\ref{thm:length} to write an algorithm to compute the class degree. However, since the bound on Theorem~\ref{thm:length} is growing exponentially with respect to the number of preimages of a symbol in $\A(Y)$, the algorithm could be hopelessly complicated. It would be ideal to find an efficient algorithm for computing the class degree.

\end{document}